\documentclass[final]{siamltex}

\usepackage{graphicx,amsmath,amsfonts,color,a4,pifont}
\usepackage{latexsym,amssymb,epsf,subfigure}

\usepackage[T1]{fontenc}
\usepackage[dvips]{epsfig}
\usepackage[dvips]{graphicx}
\usepackage{pifont}
\usepackage{ifthen}
\usepackage{float} 
\usepackage[algoruled,titlenumbered]{}
\usepackage{hyperref} 
\usepackage{cite}

\newcommand{\uha}{^{\frac{1}{2}}}
\newcommand{\umha}{^{-\frac{1}{2}}}

\newcommand{\innCnn}{\in\mathbb{C}^{n\times n}}
\newcommand{\innCrr}{\in\mathbb{C}^{r\times r}}
\newcommand{\innCrn}{\in\mathbb{C}^{r\times n}}
\newcommand{\innCnr}{\in\mathbb{C}^{n\times r}}

\newcommand{\innbCnnmr}{\in\mathbb{C}^{n\times n-r}}
\newcommand{\innCnmr}{\in\mathbb{C}^{n-r}}

\DeclareMathOperator*{\spann}{span}
\DeclareMathOperator*{\argmax}{argmax}

\newcommand{\U}{\mathcal{U}}

\newcommand{\beqo}{\begin{eqnarray*}}
\newcommand{\beq}{\begin{eqnarray}}
\newcommand{\eeqo}{\end{eqnarray*}}
\newcommand{\eeq}{\end{eqnarray}}

\newcommand{\mat}{\left[ \begin{array}{c}  }
\newcommand{\rix}{\end{array} \right] }

\numberwithin{equation}{section}

\newcommand{\im} {{\cal R}}

\newcommand{\ran} {{\cal R}}

\newcommand{\bC}{\mathbb{C}}
\newcommand{\C}{\mathbb{C}}
\newcommand{\bCn}{\mathbb{C}^n}

\newcommand{\Cn}{\mathbb{C}^n}
\newcommand{\Cnn}{\mathbb{C}^{n \times n}}
\newcommand{\Crn}{\mathbb{C}^{r \times n}}
\newcommand{\Cnr}{\mathbb{C}^{n \times r}}
\newcommand{\Cnnr}{\mathbb{C}^{n \times n-r}}

\newcommand{\inCnn}{\in \mathbb{C}^{n \times n}}

\author{
Luis Garc\'{i}a Ramos\footnotemark[1]
\and
Reinhard Nabben\footnotemark[1]
}

\title{On Optimal Algebraic Multigrid Methods}
\begin{document}

\maketitle
\renewcommand{\thefootnote}{\fnsymbol{footnote}}
\footnotetext[1]{
Technische Universit\"at Berlin, Institut f\"ur Mathematik, Stra\ss e des 17.
Juni 136, D-10623 Berlin,
Germany
 (\{garcia, nabben\}@math.tu-berlin.de). 
}

\renewcommand{\thefootnote}{\arabic{footnote}}
\begin{abstract}
In this note we present an alternative way to obtain optimal
interpolation operators for two-grid methods applied to Hermitian positive
definite linear systems.  In \cite{FalVZ05,Zik08} the $A$-norm of the error
propagation operator of algebraic multigrid methods is characterized. These
results are just recently used in \cite{XuZ17, Bra18} to determine optimal
interpolation operators. Here we use a characterization not of the $A$-norm but
of the spectrum of the  error propagation operator of two-grid methods, which
was proved in  \cite{GarKN18}. This characterization holds for arbitrary
matrices. For Hermitian positive definite systems this result   leads to
optimal interpolation operators with respect to the $A$-norm in a short way, 
moreover, it also leads to optimal interpolation operators with respect to the
spectral radius. For the symmetric two-grid method (with pre- and 
post-smoothing)
the optimal interpolation operators are the same. But for a two-grid method
with only post-smoothing the optimal interpolations (and hence the
optimal algebraic multigrid
methods) can be   different.  Moreover, using the  characterization of the
spectrum,
we can show that the found  optimal interpolation operators are also optimal
with
respect to the condition number of the multigrid preconditioned system.  
\end{abstract}



\pagestyle{myheadings}
\thispagestyle{plain}
\markboth{L. Garc\'{i}a Ramos, R. Nabben} {Optimal Algebraic Multigrid}

\section{Introduction}
Typical multigrid methods to solve the linear system 
\[
Ax = b,
\]
where $A$ is an $n \times n$ matrix, consist  of two ingredients, the smoothing
and  the
coarse grid correction. The smoothing is typically done by a
few
steps of a basic stationary iterative method, like the Jacobi or Gauss-Seidel
method.  For the coarse grid correction,
 a {\it prolongation} or {\it interpolation}
operator $P \in \Cnr$ and a   {\it
restriction} operator $R \in \Crn$  are needed. The coarse grid matrix is then
defined as
\beq \label{def:multAC}
A_C :=  RAP \innCrr.
\eeq 

Here we always assume  that $A$ and $A_C$ are  non-singular. 
The multigrid or algebraic multigrid (AMG) error
propagation matrix  is then given by
 \beq \label{mgiteration}
E_M = (I-M_2^{-1}A)^{\nu_2}(I -  PA_C^{-1}RA)(I-M_1^{-1}A)^{\nu_1},
\eeq
where $M_1^{-1} \innCnn$ and $M_2^{-1} \innCnn$  are   {\it smoothers}, $\nu_1$
and $\nu_2$  are the number of pre- and post-smoothing steps respectively, and
$PA_C^{-1}R$
is
the
{\it
coarse grid
correction} matrix. The multigrid method is convergent if and only if the
spectral radius of the
error propagation matrix
$\rho(E_m)$ is less than one.
Alternatively, the norm of the error propagation matrix $\|E_M\|$ 
can be considered, where
$\|\cdot\|$
is
a
consistent matrix
norm, and in this case one has  
\[
\rho(E_M) \leq \|E_M\|.
\]  
The aim of algebraic multigrid methods is to balance the interplay between
smoothing and coarse grid correction steps. However, most of the existing AMG
methods first fix a smoother and then optimize a certain quantity to choose
the interpolation $P$ and restriction $R$.

To simplify the analysis, we assume that there exists a non-singular matrix $X$
such that
\beq \label{mgx}
(I-X^{-1}A) = (I-M_1^{-1}A)^{\nu_1}(I-M_2^{-1}A)^{\nu_2},
\eeq
it can be shown that such a  non-singular matrix $X$ exists if the spectral
radius of $
(I-M_1^{-1}A)^{\nu_1}(I-M_2^{-1}A)^{\nu_2}$ is less  than one, see e.g.
\cite{BenS97}. Note that the matrix $E_M$ can be
written  as 
\beq \label{mgb}
E_M = I-BA,
\eeq
where the  matrix $B$ is known as the multigrid preconditioner, i.e., $B$ is an
approximation of $A^{-1}$.
Therefore,
eigenvalue estimates of $BA$ are of interest and  they lead to
estimates for
the eigenvalues of
$E_M$.


%




The following theorem, proved by Garc{\'i}a Ramos, Kehl  and Nabben in
\cite{GarKN18},
gives a characterization of the spectrum of $BA$, denoted by $\sigma(BA)$,  and
hence a
characterization of the spectrum of the general error propagation matrix $E_M$.

\begin{theorem} \label{theo:mg:eig}
Let $A \innCnn$ be  non-singular, and let   $P \innCnr $ and  $R \innCrn $ such
that $RAP$ is non-singular. Moreover, let $M_1 \innCnn$ and $M_2 \innCnn $ be
such  that  that the matrices $X$ in \eqref{mgx} and $RXP$ are  non-singular.
Then the following statements hold:
\begin{enumerate}
\item[(a)] The multigrid preconditioner $B$ in
\eqref{mgb}  is non-singular. 

\item[(b)] If $\tilde P, \tilde R \in \bC^{n \times n-r}$ are matrices
such that the columns  of
$\tilde P$ and $\tilde R $ form  orthonormal  bases of $(\im
(P))^\perp$ and
$(\im (R^{H}))^\perp$ (the orthogonal complements of $\im
(P)$ and $\im (R^{H})$ in the Euclidean inner product)  respectively,
then
the
matrices
$\tilde
P^HA^{-1}\tilde
R$ and $\tilde P^HX^{-1}\tilde R$
are
non-singular
and the spectrum of
$BA$  is given by
\[\sigma(BA) = \{1\} \cup \sigma(\tilde P^HX^{-1}\tilde R (\tilde
P^HA^{-1}\tilde
R)^{-1}).\]

\end{enumerate}
\end{theorem}

We will apply this theorem to Hermitian positive definite (HPD)  matrices to
determine
the
optimal interpolation operators of AMG methods with respect to
the
spectral radius of the error propagation matrix.
For HPD 
matrices, optimal interpolation  operators with respect to the
$A$-norm have been obtained recently in \cite{XuZ17, Bra18}.
We will show that the optimal interpolation operators with respect to the
spectral
radius
for
the
symmetric/symmetrized
multigrid
method
(with
pre-
and
post-smoothing) and
the
optimal
interpolation operator with respect to the $A$-norm are the same. But for
multigrid
with
only
a
post-smoothing step
the optimal interpolation operators with respect to the spectral radius and
$A$-norm
(and
hence
the
optimal
algebraic
multigrid
methods)
can be   different. Using Theorem \ref{theo:mg:eig} we can also show that the 
interpolation operators with respect to the spectral radius are also optimal 
with
respect to the condition number of the multigrid preconditioned system. 

\section{Optimal interpolation  for Hermitian positive definite matrices}

Let  $A \in \Cnn$ be HPD and recall that the
norm
induced
by
$A$ (or $A$-norm) is defined for $v \in \bCn$ and $S
\in \Cnn$ by
\[
\| v \|_A^2 = (v,v)_A = \|A\uha v\|_2^2,
\]
and 
\[
\| S \|_A = \|A\uha S A\umha\|_2.
\]

We will study the  following two-grid  methods given by the error
propagation
operators
\beq \label{mge}
E_{TG} = (I-M^{-H}A)(I -  PA_C^{-1}P^HA)
\eeq
 and the symmetrized version
\beq \label{smge}
E_{STG} = (I-M^{-H}A)(I -  PA_C^{-1}P^HA)(I-M^{-1}A).
\eeq
Thus we are using $R = P^H$. The range of $P$,
i.e.
$\ran
(P)$,
is
called
the
coarse space $V_c$.
We assume that the smoother $M^{-1}$ is fixed and let $E_{TG}$ and $E_{STG}$
vary
with
respect  to the choice of the interpolation operator $P$. In addition, we
assume that the smoother
$M^{-1}$
satisfies 
\[
 \|I-M^{-1}A \|_A < 1,
\]
which is equivalent to the condition
\beq \label{eq:pos}
M +  M^{H} - A  \quad \mbox{is  positive definite,} 
\eeq
see, e.g., \cite{Vas08}. Given a fixed  smoother $M^{-1}$ such that $\|
I-M^{-1}A\|_A < 1$, many AMG
methods are designed to minimize $ \|E_{TG}\|_A$ or a related quantity. We
say an interpolation operator $P^\star$ is optimal
if it minimizes  $ \|E_{TG}\|_A$. In
view of the equality
\beq \label{normeq}
\|E_{STG}\|_A = \|E_{TG}\|_A^2,
\eeq
proved by 
Falgout
and Vassilevski in
\cite{FalV04}, we can conclude that an optimal interpolation operator
$P^\star$  
also minimizes $\|E_{STG}\|_A$. Zikatanov
proved in \cite[Lemma 2.3]{Zik08} (see also \cite[Theorem 4.1]{FalVZ05}) that
\[
 \|E_{TG}\|_A^2 = 1 - \frac{1}{K(V_c)},
\]
where  $ K(V_c)$ is  a  quantity  depending  on the  coarse space, defined by

\[ K(V_c) = \sup_{v \in \Cn} \frac{\|(I-Q)v\|_{\tilde M}^2}{\|v\|_A^2}. \]
Here  $ \tilde M^{-1} = M^{-1} + M^{-H} - M^{-1} A M^{-H}$ is the symmetrized
smoother and $Q
= P(P^T\tilde M P)^{-1}\tilde M$. Although  this 
equality has been known for a long time, only  recently it was used to
determine
optimal prolongation operators formulated in terms of eigenvectors,
which lead  to a minimal  value of
$\|E_{TG}\|_A$ for a given smoother (see \cite{XuZ17, Bra18}).  We will give an
alternative proof of this result using the
characterization
of
the
eigenvalues of the multigrid iteration operator  given in Theorem
\ref{theo:mg:eig}.  

We  consider first  the  more general error propagation matrix $E_M$ in
\eqref{mgiteration} with $R= P^H$  and  $E_{M} = I - BA$. Let $\U=
\mathcal{R}(P) $ be 
the range of the interpolation operator
$P \in \Cnr$, and $\tilde U \in \C^{n \times n-r} $ be a matrix with
orthonormal columns that span
$\U
^\perp$ (the
orthogonal complement  of $\U$ with respect to the Euclidean inner product).
Then
Theorem \ref{theo:mg:eig}  leads to
\[
\sigma (BA) = \{1\} \cup \sigma(\tilde U^HX^{-1}\tilde U (\tilde
U^HA^{-1}\tilde U)^{-1}).
\]

In what follows, given a matrix $C \inCnn$ with real eigenvalues we will denote
by
$\lambda_{\max}(C)$ and
$\lambda_{\min}(C)$
the maximum and minimum eigenvalues of $C$ respectively.

Assuming that $X$ is
Hermitian positive definite and that
$\lambda_{\max}(BA)$ is at most one, we have
$\rho(E_M) = 1 - \lambda_{\min}(BA)$. In order to find an optimal interpolation
operator for the error propagation matrix,  we need  to first find
\[  \tilde{U}^\star \in \argmax_{\tilde U \innbCnnmr,\, \tilde
U^H\tilde U = I}
\lambda_{\min}(\tilde
U^HX^{-1}\tilde
U
(\tilde
U^HA^{-1}\tilde U)^{-1}),
\]



and then find an interpolation operator $P^\star \in \Cnr$ such that
$\im(P^\star) =\im(\tilde U^\star)^{\perp}$. The following lemma  solves the 
first problem.

\begin{lemma} \label{theo:lemma1}
Let $A, X \innCnn$ be Hermitian positive definite and let
$\{(\mu_i,w_i)\}_{i=1}^n$ be the eigenpairs of the generalized eigenvalue
problem
 \[X^{-1}w = \mu
A^{-1}w,\]
where
\beq
0 < \mu_1 \leq \mu_2 \leq \ldots \leq  \mu_n.
\eeq
Then
\[\max_{\tilde U \innbCnnmr, \, \tilde U^H \tilde U=I} \lambda_{\min} (\tilde
U^HX^{-1}\tilde U
(\tilde
U^HA^{-1}\tilde U)^{-1}) = \mu_{r+1}
\]
which is achieved by 
\[
\tilde W = [\tilde{w}_{r+1}, \ldots, \tilde{w}_n],
\in \C^{n-r}\] 
where the columns of $\tilde W$ are orthogonal in the Euclidean inner 
product and satisfy
$\spann\{\tilde{w}_i\}_{i=1}^n = \spann\{w_i\}_{i=1}^n$.
\end{lemma}

\begin{proof}
Let $\tilde{U} \in \C^{n \times n-r }$ with $\tilde{U}^H\tilde{U} = I$. By the
Courant-Fischer theorem we obtain 
\begin{align*}
\lambda_{\min} (\tilde U^HX^{-1}\tilde U (\tilde U^HA^{-1}\tilde U)^{-1}) & =
\min_{\substack{z \in \Cn \\ z \neq 0}} \frac{z^H \tilde U X^{-1} \tilde U z
}{
z^H
\tilde U^H A^{-1} \tilde Uz}\\
& = \min_{ \substack{ z \in \im(\tilde U) \\ z \neq 0}}\frac{z^{H}
X^{-1}z}{z^HA^{-1}z},
\end{align*}

Thus, if $\mathbf{V}$ is the set of subspaces of $\C^n$ of dimension $n-r$, we
have
\[\max_{\tilde{U} \in \Cnnr, \, \tilde{U}^H\tilde U = I }
\lambda_{\min} (\tilde U^HX^{-1}\tilde U (\tilde U^HA^{-1}\tilde U)^{-1}) =
\max_{\tilde{\mathcal{U}} \in \mathbf V } \min_{ \substack{ z \in
\tilde{\mathcal{U}}
\\
z \neq 0}}
\frac{z^H X^{-1} z}{z^HA^{-1}z}  = \mu_{r+1},
\]
and the maximum is attained by choosing a matrix $\tilde W =
[\tilde w_{r+1}, \ldots, \tilde w _{n}]$ such that 
the columns of $\tilde W$ are orthogonal in the Euclidean inner 
product and satisfy
$\spann\{\tilde{w}_i\}_{i=1}^n = \spann\{w_i\}_{i=1}^n$. 
\end{proof}

The previous lemma is the main tool to obtain the optimal interpolation
operators.

\begin{theorem}\label{theo:main}
Let $A \innCnn$ and $ X \innCnn$ as in \eqref{mgx} be Hermitian positive
definite. Let $\{(\lambda_i, u_i)\}_{i=1}^n$ be the
eigenpairs
of $X^{-1}A$, where $
\lambda_1 \leq \lambda_2 \leq \ldots \leq  \lambda_n $,  and suppose that
$\lambda_{\max}(BA) \leq 1$. Then
\beq
\min_{\substack{P \in \Cnr \\ \rank(P)=r}} \rho(E_{M}) =  1 - \min_{\substack{P
\in \Cnr \\ \rank(P)=r}} \lambda_{\min}(BA) = 1 -
\lambda_{r+1}.
\eeq
An optimal interpolation operator is given by 
\begin{equation} \label{eq:popt}
P_{\mathrm{opt}} = [u_{1}, \ldots , u_r].
\end{equation}
\end{theorem}
\begin{proof}
Since $\lambda_{\max}(BA) \leq 1$,  we have that 
\beqo
\rho(E_{M}) =  1 - \lambda_{\min}(BA).
\eeqo
Note that the eigenvalues $\lambda_i $ are  the same as the $\mu_i$ in Lemma
\ref{theo:lemma1}.
According to Lemma  \ref{theo:lemma1}, we need  to find vectors which are
orthogonal
to
the eigenvectors  $w_{r+1}, \ldots , w_n$ of the generalized eigenvalue problem
$X^{-1}w = \mu A^{-1}w$. Now, consider the 
vectors $\{u_i\}_{i=1}^r$.
The $u_i$ are also eigenvectors of the generalized eigenvalue problem  $Au =
\lambda Xu$. Moreover, the vectors  $Xu_i = w_i$  are  eigenvectors  of the
generalized
eigenvalue problem $X^{-1}w = \mu A^{-1}w$. But the $w_i$ are
$X^{-1}$-orthogonal (the $X\umha w_i$ are eigenvectors of the Hermitian matrix
$X\uha A^{-1} X\uha$). Thus, the $u_i$, $i = 1, \ldots, r$ are  orthogonal to
the  $w_{r+1}, \ldots , w_n$ in the Euclidean inner product and the
interpolation operator
$P_{\mathrm{opt}}$ given by \eqref{eq:popt} is the corresponding
minimizer. \end{proof}

Now, we consider $E_{TG}$ and $E_{STG}$  defined in \eqref{mge}  and
\eqref{smge}. Again   $E_{STG}$ and $E_{TG}$    can be written   as
\beqo
E_{STG} & = & I - B_{STG}A, \\
E_{TG} & = & I - B_{TG}A,
\eeqo
for some  matrices $B_{STG}$ and $B_{TG}$ in $\Cnn$. A straightforward
computation shows  that  $B_{STG}$  is Hermitian, and by 
\cite[Lemma 2.11]{Ben01} we have
\beq \label{ben}
\|E_{STG}\|_A = \|I - B_{STG}A\|_A = \rho(I - B_{STG}A).
\eeq
Moreover, the maximal eigenvalue of $B_{STG}A$ satisfies
$\lambda_{\max}(B_{STG}A) \leq 1$, see e.g. \cite[Theorem 3.16]{Vas08}.
We
then
obtain
\[
\|E_{TG}\|_A^2 = \|E_{STG}\|_A = \rho(I - B_{STG}A) = 1 -
\lambda_{min}(B_{STG}A).
\]

The matrix $X$ in \eqref{mgx} is given by 
\beq \label{defX}
X^{-1}_{STG} = M^{-1} +  M^{-H} - M^{-1} AM^{-H} = M^{-1}( M +  M^{H} -
A)M^{-H}.
\eeq
With \eqref{eq:pos} we have  that $X_{STG}$ is Hermitian positive definite.
 We obtain the following corollary.
\begin{corollary} \label{coro:one}
Let  $A\inCnn$  be Hermitian positive definite. Let $ M \inCnn$ such $M + M^H -
A$ is Hermitian positive definite, and
let $X_{STG}^{-1}$  be as in \eqref{defX},  and let $\{(\lambda_i,
u_i)\}_{i=1}^n$
be
the eigenpairs
of $X_{STG}^{-1}A$, where $
\lambda_1 \leq \lambda_2 \leq \ldots \leq  \lambda_n $, 
 Then
\beq
\min_{\substack{P \in \Cnr \\ \rank(P)=r}} \|E_{STG}\|_A =
\min_{\substack{P \in \Cnr \\ \rank(P)=r}}\rho(E_{STG}) = 
\min_{\substack{P \in \Cnr \\ \rank(P)=r}}\|E_{TG}\|_A^2 = 1 -
\lambda_{r+1}.
\eeq
An optimal interpolation operator is given by 
\[
P_{\mathrm{opt}} = [v_{1}, \ldots , v_r].
\]
\end{corollary}
\begin{proof}
We have  that $X_{STG}$ is positive definite and $\lambda_{\max}(B_{STG}A) \leq
1$. By Theorem \ref{theo:main} we obtain the desired result.
\end{proof}

Next, let us consider  the non-symmetric multigrid method defined implicitly
by $E_{TG}$, in \eqref{mge}.
We use a Hermitian positive  definite smoother $M^{-1}$. The matrix $X$ in
\eqref{mgx} is given by
\beq \label{defXtg}
X^{-1}_{TG} = M^{-1}.
\eeq

Hence
\beqo
\rho(E_{TG}) = 1 - \lambda_{\min}(B_{TG}A)
\ \  \mbox{or} \ \ 
\rho(E_{TG}) = -(1 - \lambda_{\max}(B_{TG}A)).
\eeqo

Therefore, it is not clear which of $\lambda_{\min}(B_{TG}A)$ 
or $\lambda_{\max}(B_{TG}A)$ equals the spectral radius.
One way to overcome  this problem is scaling. Note that we  have for all
Hermitian positive defnite matrices $X$ and $A$ and for all matrices $\tilde U
\in \Cnnr$
\beqo
\lambda_{\max}(\tilde U^HX^{-1}\tilde U (\tilde U^HA^{-1}\tilde
U)^{-1})
& = &  \max_{z \innCnmr} \frac{z^H\tilde U^HX^{-1}\tilde Uz}{ z^H\tilde
U^HA^{-1}\tilde Uz}\\
& = &  \max_{\tilde z \in {\ran (\tilde U)} } \frac{\tilde z^HX^{-1}\tilde z}{
\tilde z^HA^{-1}\tilde z}\\
& \leq &\max_{\tilde z \in \bCn  } \frac{\tilde z^HX^{-1}\tilde z}{
\tilde z^HA^{-1}\tilde z}\\
& = & \lambda_{\max}(X^{-1}A).
\eeqo

Hence, the  Hermitian smoother
\beqo
\hat M^{-1} = \frac{1}{\lambda_{\max}(M^{-1}A)}M^{-1}
\eeqo
satisfies
\beq \label{eq:spec1}
\lambda_{\max}(\hat M^{-1}A) = 1.
\eeq
With Theorem \ref{theo:mg:eig} and $X^{-1} = \hat M^{-1}$ we then have 

\beqo
\lambda_{\max} ((B_{TG}A) = 1,
\eeqo
thus
\beqo
\rho(E_{TG}) = 1 - \lambda_{\min } (B_{TG}A).
\eeqo

Note that \eqref{eq:spec1} is equivalent to  $\hat M - A $ being positive
semidefinite. This discussion leads to the following corollary.
\begin{corollary} \label{coro:two}
Let  $A\inCnn$  be Hermitian positive definite. Let $ M \inCnn$ such $M - A$ is
Hermitian positive definite.
Let $X_{TG}^{-1} =  M^{-1}$.   
 Let $
\tilde \lambda_1 \leq \tilde \lambda_2 \leq \ldots \leq  \tilde \lambda_n $
be the  eigenvalues of $X_{TG}^{-1}A$  and let $x_i$, $i = 1, \ldots, n$, be
the corresponding eigenvectors. Then
\beq \label{eq:min.case2}
\min_{\substack{P \in \Cnr \\ \rank(P)=r}}\rho(E_{TG}) = 1 - \tilde
\lambda_{r+1}.
\eeq
An optimal interpolation operator is given by 
\beq  \label{eq:min.case2int}
P_{\mathrm{opt}} = [x_{1}, \ldots , x_r].
\eeq
\end{corollary}
\begin{proof}
The matrix $X_{TG}^{-1} = M^{-1}$ is Hermitian positive definite. Moreover,
since $M - A$ is also Hermitian positive definite the eigenvalues of
$X_{TG}^{-1}A$ are less then  one. Thus, with Theorem \ref{theo:mg:eig},
$\lambda_{\max}(B_{TG}A) = 1$.  So, with  Theorem \ref{theo:main}  we obtain
\eqref{eq:min.case2} and \eqref{eq:min.case2int}.
\end{proof}

Now we will compare the optimal interpolation with respect to the $A$-norm as
given in Corollary \ref{coro:one}, with  the optimal interpolation with respect
to the spectral radius as given in Corollary \ref{coro:two}. Using $M=M^H$ and
$M - A$ Hermitian positive definite, the vectors used in  Corollary
\ref{coro:one}
are  eigenvectors of
\beqo
X^{-1}_{STG}A = 2M^{-1}A - M^{-1}AM^{-1}A,
\eeqo
while in Corollary \ref{coro:one} we use  the eigenvectors of
\beqo
X^{-1}_{TG}A = M^{-1}A.
\eeqo
But $X^{-1}_{STG}A$ is just a polynomial in $M^{-1}A$ , where   the polynomial
is given by
\beq \label{eq:pol}
p(t) = 2t - t^2.
\eeq
Thus, the eigenvectors of both matrices are the same. Moreover, the
eigenvalues are  related
by   the above polynomial. Hence, the eigenvectors corresponding  to the
smallest eigenvalues of
$X^{-1}_{STG}A$  are the same   eigenvectors that correspond to the smallest
eigenvalues of $X^{-1}_{TG}A$. In consequence, the optimal interpolation in
Corollary
\ref{coro:one}  and Corollary
\ref{coro:two} are the same, if we assume that $M - A$ is Hermitian positive
definite.

Next, let us have  a closer look to the non-symmetric two-grid method and
avoid
scaling. We assume  that
the smoother $M$ is Hermitian  and  leads to a convergent scheme, i.e.
\beq  \label{eq:smoother:con:}
\rho(I - M^{-1}A) < 1, 
\eeq
which implies $\sigma(M^{-1}A) \subset (0,2).$ Thus, for the matrix $E_{TG}$
we have as above
\beqo
\rho(E_{TG}) = 1 - \lambda_{\min}(B_{TG}^{-1}A) < 1 
\ \ \mbox{or} \ \
\rho(E_{TG}) = -(1 - \lambda_{\max}(B_{TG}^{-1}A)) < 1.
\eeqo

Let 
\beqo
Z = \tilde U^HX_{TG}^{-1}\tilde U (\tilde U^HA^{-1}\tilde U)^{-1}.
\eeqo

Then we have $\sigma(Z) \subset (0,2)$ and with  Theorem \ref{theo:mg:eig}
\beqo
\sigma(E_{TG}) = \{0\} \cup \sigma(I-Z).
\eeqo
But $\sigma(I-Z) \subset (-1,1) $. To get an upper bound for  the minimal
spectral radius of $E_{TG}$
over all interpolation we consider the matrix $(I - Z)^2$. Our next theorem
deals with this case.

\begin{theorem} \label{theo:main2}
Let  $A\inCnn$  be Hermitian positive definite, and let $ M \inCnn$ be
Hermitian
such that $\rho(I - M^{-1}A) < 1$.
Let $X_{TG}^{-1} =  M^{-1}$,
and let $\{(\lambda_i,y_i)\}_{i=1}^n$ be the eigenpairs of $(I -
X_{TG}^{-1}A)^2$ with $
\hat \lambda_1 \leq \hat \lambda_2 \leq \ldots \leq  \hat \lambda_n $. Then
\beq \label{eq:min.case3}
\min_{\substack{P \in \Cnr \\ \rank(P)=r}}\rho(E_{TG}) \leq (\hat
\lambda_{n-r})^{\frac{1}{2}}.
\eeq
The spectral radius $(\hat \lambda_{n-r})^{\frac{1}{2}}$ can be achieved by the
interpolation operator
\beq  \label{eq:min.case3int}
\hat P = [y_{n-r+1}, \ldots , y_n].
\eeq
\end{theorem}
\begin{proof}
The proof follows immediately from the above arguments. 
\end{proof}


Note that the above Theorems correspond to
clear statements: the optimal interpolation operators are given
by
those
eigenvectors of $X^{-1}A$ for which the smoothing is slowest to converge.

\section{The optimal interpolation with respect to the  condition number}

Note  that for symmetric multigrid where  $M + M^H - A$ is Hermitian  positive
definite  the largest eigenvalue of
$B_{STG}A$ is one (see e.g. \cite{Not15}).  As seen in the proof of Corollary
\ref{coro:two}, the same holds  for  $B_{TG}A$ when we assume  that  $M - A$ is
Hermitian  positive definite. The later
assumption can be obtained  by scaling, however, this scaling affects the
spectral radius of the
error propagation matrix. But for the condition number of the multigrid
preconditioned system, this scaling has no effect.

Theorem \ref{theo:mg:eig} characterizes the  spectrum of $B_{STG}A$ and
$B_{TG}A$. Following the arguments above, where   we found optimal
interpolation operators, such that
$\lambda_{\min}(B_{STG}A)$ and $\lambda_{\min}(B_{TG}A)$ are maximal, we obtain
that the same interpolation operators are optimal with respect to the condition
number $\kappa$ of the preconditioned system. This leads to the next result.

 \begin{theorem}
Let  $A\inCnn$  be Hermitian positive definite. Let $ M \inCnn$ such $M + M^H -
A$ is Hermitian positive definite.
Let $X_{STG}^{-1}$  be as in \eqref{defX}.  
 Let $\{(\lambda_i,v_i)\}_{i=1}^n$ be the eigenpairs of $X_{STG}^{-1}A$, where
$
\lambda_1 \leq \lambda_2 \leq \ldots \leq  \lambda_n $. Then
\beq
\min_{\substack{P \in \Cnr \\ \rank(P)=r}} \kappa(B_{STG}A) =
\frac{1}{\lambda_{r+1}}.
\eeq
An optimal interpolation operator is given by 
\[
P_{\mathrm{opt}} = [v_{1}, \ldots , v_r].
\]
\end{theorem}
  
Our final result gives the optimal interpolation operator for the non-symmetric
two-grid
method with respect to the condition number $\kappa$.

\begin{theorem}
Let  $A\inCnn$  be Hermitian positive definite. Let $ M \inCnn$ be Hermitian
positive definite  such that $\rho(I - M^{-1}A) < 1.$
Let $X_{TG}^{-1} =  M^{-1}$, and let $\{(\tilde \lambda_i, x_i)\}_{i=1}^n$ be 
the eigenpairs of $X_{TG}^{-1}A$  where  $
\tilde \lambda_1 \leq \tilde \lambda_2 \leq \ldots \leq  \tilde \lambda_n $.
Then
\beqo
\min_{\substack{P \in \Cnr \\ \rank(P)=r}}  \kappa(B_{TG}A) = \frac{\tilde
\lambda_{n}}{\tilde
\lambda_{r+1}}
\eeqo
An optimal interpolation operator is given by 
\beqo
P_{\mathrm{opt}} = [x_{1}, \ldots , x_r].
\eeqo
\end{theorem}

Note, that  in all cases of the previous sections any other interpolation
operator $\tilde P$  with  $\ran (\tilde P) = \ran (P_{\mathrm{opt}})$ is also
optimal.

\section{Conclusion}
As mentioned in \cite{XuZ17}, the  $A$ in AMG methods can also be understood as
an
$A$
for
Abstract 
Multigrid Methods. Here  we contributed to the  theory of  
abstract multigrid methods by presenting alternate derivations of previously
known results and by establishing new results. Building on a  result from
\cite{GarKN18} which gives a
characterization of the spectrum of the
error propagation operator and the preconditioned system of two-grid methods,
we derived optimal interpolation operators with respect  to the $A$-norm and
the spectral radius of the  error propagation operator matrix in a
short way. 
We also showed that these interpolation operators
are optimal  with respect to the condition number of the preconditioned system.

\bibliographystyle{siamplain}
\bibliography{GarKN2.bib}

\end{document}